\documentclass[11pt]{amsart}

\usepackage[foot]{amsaddr}

\usepackage{amssymb,amsfonts,amsmath,graphicx,verbatim,eufrak,amsthm,calc}
\usepackage{color}
\usepackage{xcolor}
\usepackage{newfloat}
\usepackage{hyperref}
\usepackage{tikz}
\hypersetup{
hidelinks,backref=true,pagebackref=true,hyperindex=true,
    colorlinks=true,       % false: boxed links; true: colored links
    linkcolor=violet,          % color of internal links (change box color with linkbordercolor)
    citecolor=green,        % color of links to bibliography
    filecolor=magenta,      % color of file links
    urlcolor=cyan           % color of external links
}

\colorlet{genial}{black} 
\colorlet{genialsol}{black}

\makeatletter

\newtheoremstyle{genialnumbox}% % Theorem style name
{7pt}% Space above
{7pt}% Space below
{\normalfont}% % Body font
{}% Indent amount
{\small\bf\sffamily\color{genial}}% % Theorem head font
{\;}% Punctuation after theorem head
{0.25em}% Space after theorem head
{%
{\small\sffamily\color{genial}\thmname{#1}}%
{\nobreakspace\thmnumber{\@ifnotempty{#1}{}\@upn{#2}}}% Theorem text (e.g. Theorem 2.1)
\thmnote{{\nobreakspace\the\thm@notefont\sffamily\bfseries\color{black}\nobreakspace(#3)}} % Optional theorem note
}

\newtheoremstyle{blacknumex}% Theorem style name
{7pt}% Space above
{7pt}% Space below
{\normalfont}% Body font
{} % Indent amount
{\small\bf\sffamily}% Theorem head font
{\;}% Punctuation after theorem head
{0.25em}% Space after theorem head
{%
{\small\sffamily\color{genial}\thmname{#1}}%
{\nobreakspace\thmnumber{\@ifnotempty{#1}{}\@upn{#2}}}% Theorem text (e.g. Theorem 2.1)
\thmnote{{\nobreakspace\the\thm@notefont\sffamily\bfseries\color{black}\nobreakspace(#3)}} % Optional theorem note
}

\newtheoremstyle{blacknumbox} % Theorem style name
{7pt}% Space above
{7pt}% Space below
{\normalfont}% Body font
{}% Indent amount
{\small\bf\sffamily}% Theorem head font
{\;}% Punctuation after theorem head
{0.25em}% Space after theorem head
{%
{\small\sffamily\color{genial}\thmname{#1}}%
{\nobreakspace\thmnumber{\@ifnotempty{#1}{}\@upn{#2}}}% Theorem text (e.g. Theorem 2.1)
\thmnote{{\nobreakspace\the\thm@notefont\sffamily\bfseries\color{black}\nobreakspace(#3)}} % Optional theorem note
}

\newtheoremstyle{genialnum}% % Theorem style name
{7pt}% Space above
{7pt}% Space below
{\normalfont}% % Body font
{}% Indent amount
{\small\bf\sffamily\color{genial}}% % Theorem head font
{\;}% Punctuation after theorem head
{0.25em}% Space after theorem head
{%
{\small\sffamily\color{genial}\thmname{#1}}%
{\nobreakspace\thmnumber{\@ifnotempty{#1}{}\@upn{#2}}}% Theorem text (e.g. Theorem 2.1)
\thmnote{{\nobreakspace\the\thm@notefont\sffamily\bfseries\color{black}\nobreakspace(#3)}} % Optional theorem note
}

\makeatother

\RequirePackage[framemethod=default]{mdframed} % Required for creating the theorem, definition, exercise and corollary boxes

\newmdenv[skipabove=7pt,
skipbelow=7pt,
rightline=false,
leftline=false,
topline=false,
bottomline=false,
backgroundcolor=black!5,
linecolor=genial,
innerleftmargin=5pt,
innerrightmargin=5pt,
innertopmargin=10pt,
leftmargin=0cm,
rightmargin=0cm,
innerbottommargin=10pt]{tBox}

\newmdenv[skipabove=7pt,
skipbelow=7pt,
rightline=false,
leftline=false,
topline=false,
bottomline=false,
backgroundcolor=genial!10,
linecolor=genial,
innerleftmargin=5pt,
innerrightmargin=5pt,
innertopmargin=5pt,
innerbottommargin=5pt,
leftmargin=0cm,
rightmargin=0cm,
linewidth=4pt]{eBox}	

\newmdenv[skipabove=7pt,
skipbelow=7pt,
rightline=false,
leftline=true,
topline=false,
bottomline=false,
linecolor=genial!50,
innerleftmargin=5pt,
innerrightmargin=5pt,
innertopmargin=5pt,
leftmargin=0cm,
rightmargin=0cm,
linewidth=4pt,
innerbottommargin=5pt]{dBox}	

\newmdenv[skipabove=7pt,
skipbelow=7pt,
rightline=false,
leftline=false,
topline=false,
bottomline=false,
linecolor=gray,
backgroundcolor=black!5,
innerleftmargin=5pt,
innerrightmargin=5pt,
innertopmargin=5pt,
leftmargin=0cm,
rightmargin=0cm,
linewidth=4pt,
innerbottommargin=5pt]{cBox}

\newmdenv[skipabove=7pt,
skipbelow=7pt,
rightline=false,
leftline=false,
topline=false,
bottomline=false,
linecolor=gray,
backgroundcolor=black!5,
innerleftmargin=5pt,
innerrightmargin=5pt,
innertopmargin=5pt,
leftmargin=0cm,
rightmargin=0cm,
linewidth=4pt,
innerbottommargin=5pt]{pBox}

\newmdenv[skipabove=7pt,
skipbelow=7pt,
rightline=false,
leftline=false,
topline=false,
bottomline=false,
linecolor=genialsol,
innerleftmargin=5pt,
innerrightmargin=5pt,
innertopmargin=0pt,
leftmargin=0cm,
rightmargin=0cm,
linewidth=4pt,
innerbottommargin=0pt]{solBox}	

\theoremstyle{genialnumbox}
\newtheorem{thm1}{Theorem}%[section]
\newtheorem{ithm1}[thm1]{$\star$ THEOREM}
\newtheorem{ques1}[thm1]{Question}
\newtheorem{conj1}[thm1]{Conjecture}

\theoremstyle{blacknumex}
\newtheorem{exer}[thm1]{Exercise}
\newtheorem{exer*}[thm1]{$\ast$ Exercise}

\theoremstyle{blacknumbox}
\newtheorem{dfn1}[thm1]{Definition}

\theoremstyle{genialnum}
\newtheorem{cor1}[thm1]{Corollary}
\newtheorem{prop1}[thm1]{Proposition}
\newtheorem{lem1}[thm1]{Lemma}

\newtheorem{exm1}[thm1]{Example}

\newenvironment{thm}{\paragraph{ } \begin{tBox}\begin{thm1}}{\end{thm1}\end{tBox}}

\newenvironment{exe*}{\paragraph{ } \begin{eBox}\begin{exer*}}{\hfill{\color{genial}%\tiny
\ensuremath{\diamond\diamond\diamond}}\end{exer*}\end{eBox}}
\newenvironment{dfn}{\paragraph{ } \begin{dBox}\begin{dfn1}}{\end{dfn1}\end{dBox}}	

\newenvironment{cor}{\paragraph{ } \begin{cBox}\begin{cor1}}{\end{cor1}\end{cBox}}

\newenvironment{prop}{\paragraph{ } \begin{pBox}\begin{prop1}}{\end{prop1}\end{pBox}}	
\newenvironment{lem}{\paragraph{ } \begin{pBox}\begin{lem1}}{\end{lem1}\end{pBox}}

\newenvironment{lem*}[1]{\vspace{1ex}\noindent
{\bf Lemma* (#1).} [restatement]  \hspace{0.5em} \em }{ }

\newenvironment{thm*}[1]{\begin{cBox}
\vspace{1ex}\noindent 
{\bf Theorem* (#1).} [restatement]  \hspace{0.5em} }{\end{cBox}}

\theoremstyle{genialnum}

\newtheorem*{clm*}{Claim}

\newenvironment{sol}%
{\begin{solBox}
\par \noindent %\dvd 
\scriptsize
{\bf Solution to ex:{\color{blue} \arabic{exer}}.}  {\color{red} \ \  :( } \\ }%
{\hfill {\color{blue} :) $\checkmark$} \end{solBox}}
\newcommand{\ENDEXER}{
{\expandafter\comment}
{\expandafter\endcomment}
}

\newtheorem{rem}[thm1]{Remark}

\makeatletter
\renewcommand{\@seccntformat}[1]{\llap{\textcolor{genial}{\csname the#1\endcsname}\hspace{1em}}}                    
\renewcommand{\section}{\@startsection{section}{1}{\z@}
{-4ex \@plus -1ex \@minus -.4ex}
{1ex \@plus.2ex }
{\normalfont\large\sffamily\bfseries}}
\renewcommand{\subsection}{\@startsection {subsection}{2}{\z@}
{-3ex \@plus -0.1ex \@minus -.4ex}
{0.5ex \@plus.2ex }
{\normalfont\sffamily\bfseries}}
\renewcommand{\subsubsection}{\@startsection {subsubsection}{3}{\z@}
{-2ex \@plus -0.1ex \@minus -.2ex}
{.2ex \@plus.2ex }
{\normalfont\small\sffamily\bfseries}}                        
\renewcommand\paragraph{\@startsection{paragraph}{4}{\z@}
{-2ex \@plus-.2ex \@minus .2ex}
{.1ex}
{\normalfont\small\sffamily\bfseries}}

\makeatother

\newcommand{\IP}[1]{\left\langle #1 \right\rangle}
\newcommand{\iP}[1]{\langle #1 \rangle}

\newcommand{\set}[1]{\left\{#1\right\}}

\newcommand{\Integer}{\mathbb{Z}}

\newcommand{\Z}{\Integer}

\newcommand{\R}{\mathbb{R}}
\newcommand{\C}{\mathbb{C}}

\newcommand{\eps}{\varepsilon}

\newcommand{\ie}{{\em i.e.\ }}
\newcommand{\eg}{{\em e.g.\ }}

\newcommand{\1}[1]{\mathbf{1}_{\set{ #1 } }}

\def\squareforqed{\hbox{\rlap{$\sqcap$}$\sqcup$}}
\def\qed{\ifmmode\squareforqed\else{\unskip\nobreak\hfil
\penalty50\hskip1em\null\nobreak\hfil\squareforqed
\parfillskip=0pt\finalhyphendemerits=0\endgraf}\fi}

\newcommand{\ignore}[1]{ }

\newcommand{\dist}{\mathrm{dist}}
\newcommand{\vphi}{\varphi}

\newcommand{\Ee}{\mathcal{E}}

\newcommand{\T}{\mathbb{T}}

\newcommand{\define}[1]{\textbf{#1}}

\renewcommand{\H}{\mathcal{H}}

\newcommand{\supp}{\mathrm{supp}}

\newcommand\blfootnote[1]{%
  \begingroup
  \renewcommand\thefootnote{}\footnote{#1}%
  \addtocounter{footnote}{-1}%
  \endgroup
}

\begin{document}

\title[Inequalities for Dirichlet eigenvalues]{Universal inequalities for Dirichlet eigenvalues on discrete groups}

\author{Bobo Hua}
\address{Bobo Hua: School of Mathematical Sciences, LMNS,
Fudan University, Shanghai 200433, China; Shanghai Center for
Mathematical Sciences, Fudan University, Shanghai 200438,
China.}
\email{bobohua@fudan.edu.cn}

\author{Ariel Yadin}
\address{Ariel Yadin: Ben-Gurion University of the Negev. }
\email{yadina@bgu.ac.il}

\begin{abstract}
We prove universal inequalities for Laplacian eigenvalues 
with Dirichlet boundary condition on subsets of certain discrete groups.
The study of universal inequalities on Riemannian manifolds was initiated by Weyl, Polya, Yau, and others. Here we focus on a version by Cheng and Yang.

Specifically, we prove Yang-type universal inequalities for Cayley graphs of 
finitely generated amenable groups, as well as for the $d$-regular tree 
(simple random walk on the free group).
\end{abstract}

\maketitle %

\blfootnote{
\textbf{Acknowledgements.}
We thank the helpful discussions and suggestions on universal inequalities on graphs by Yong Lin.

B.H.\ is supported by NSFC, no.\ 11831004 and no.\ 11926313.
A.Y.\ is partially supported by the Israel Science Foundation (grant no.\ 1346/15).
}

\section{Introduction}

The spectral theory of Laplace-Beltrami operators on Riemannian manifolds was extensively studied in the literature, see e.g. \cite{CouHil53,Chavel84,SY94,Libook12}. For a bounded domain $\Omega$ in a Riemannian manifold, 
we denote by
$$0<\lambda_1<\lambda_2\leq\lambda_3\leq\cdots\ \uparrow \infty$$ the spectrum of the Laplace-Beltrami operator with Dirichlet boundary condition on $\Omega$, counting the multiplicity of eigenvalues.

For the Euclidean space, Weyl \cite{Weyl} proved the asymptotic behavior of eigenvalues that
$$\lambda_k\sim\frac{4\pi^2}{\left(\omega_n\mathrm{vol}(\Omega)\right)^{\frac{2}{n}}}k^{\frac{2}{n}},\quad k\rightarrow\infty,$$
where $\omega_n$ is the volume of the unit ball in $\R^n$ and $\mathrm{vol}(\Omega)$ is the volume of $\Omega$. It was conjectured by P\'{o}lya \cite{Pol61} that
$$\lambda_k\geq\frac{4\pi^2}{(\omega_n\mathrm{vol}(\Omega))^{\frac{2}{n}}}k^{\frac{2}{n}},\quad k=1,2,3,\cdots.$$
Li and Yau \cite{LY83} proved that
$$\lambda_k\geq\frac{n}{n+2}\frac{4\pi^2}{(\omega_n\mathrm{vol}(\Omega))^{\frac{2}{n}}}k^{\frac{2}{n}},\quad
k=1,2,3,\cdots.$$

Payne, Polya and Weinberger \cite{PPW} proved the gap estimate of consecutive eigenvalues for a bounded domain in $\R^2,$ generalized to $\R^n$ by Thompson \cite{Thomp69}, that for any $k\geq 1,$
$$\lambda_{k+1}-\lambda_k\leq\frac{4}{nk}\sum_{i=1}^k\lambda_i.$$ This was improved by Hile and Protter \cite{HP80}. A sharp inequality was proved by Yang \cite{Yang91,CY07} that
\begin{equation}\label{eq:Yangint1}\sum_{i=1}^k(\lambda_{k+1}-\lambda_i)^2\leq\frac{4}{n}\sum_{i=1}^k\lambda_i(\lambda_{k+1}-\lambda_i).\end{equation} 
As is well-known, see \eg \cite{Ash99}, Yang's inequality implies the Payne-Polya-Weinberger inequality etc. These are called universal inequalities for eigenvalues since they are independent of the domain $\Omega.$ See \cite{AshBen1991,AshBen92,AshBen1994,AshBen96,HarStu97,Ash99,Ash02,ChengYang05,AshBen2007} for more results regarding Euclidean spaces.

Universal inequalities have been generalized to eigenvalues of Laplace-Beltrami operators on Riemannian manifolds. In particular, Yang's inequality has been proved for space forms. For the unit $n$-sphere, Cheng and Yang \cite{ChengYang05} proved that
$$\sum_{i=1}^k(\lambda_{k+1}-\lambda_i)^2\leq\frac{4}{n}\sum_{i=1}^k(\lambda_{k+1}-\lambda_i)(\lambda_i+\frac{n^2}{4}).$$ For $\mathbb{H}^n,$ the $n$-dimensional hyperbolic space of sectional curvature $-1,$ Cheng and Yang \cite{ChengYang09} proved that
\begin{equation}\label{eq:hyp1}\sum_{i=1}^k(\lambda_{k+1}-\lambda_i)^2\leq 4\sum_{i=1}^k(\lambda_{k+1}-\lambda_i)(\lambda_i-\frac{(n-1)^2}{4}).
\end{equation} Note that $\frac{(n-1)^2}{4}$ is the bottom of the spectrum of $\mathbb{H}^n.$ For a general Riemannian manifold, Chen and Cheng \cite{ChenCheng08} proved a variant of Yang's inequality using related geometric quantities via isometric embedding into the Euclidean space. For universal inequalities on manifolds, we refer the readers to  \cite{Li80,YangYau80,Leung91,Har93,HarMich94,ChengYang06JPS,Har07,SCY08,ChengYang09,EHI09,CZL12,ChengPeng13,CZY16}.

In this paper, we study universal inequalities for eigenvalues on graphs, in particular Cayley graphs of discrete groups. We recall the setting of general networks. A \define{network} is a pair $(V,c)$ where 
$V$ is a countable set and $c: V \times V \to [0,\infty)$
is called the \define{conductance}.
The conductance must satisfy $0 \leq c(x,y) = c(y,x) < \infty$ (symmetric) and 
and $\pi(x) := \sum_y c(x,y) < \infty$ for every $x$.  
We write $x \sim y$ to indicate $c(x,y)>0$ (in which case we say that $x \sim y$ is an {\em edge} in the network).
A network naturally provides a {\em reversible Markov chain}, whose transition matrix is given by 
$P(x,y) = \frac{c(x,y)}{\pi(x)}$. The (normalized) \define{Laplacian} is the operator $\Delta = I - P,$ where $I$ denotes the identity operator, \ie $$\Delta f(x) = \sum_y P(x,y) (f(x)-f(y)).$$ 
We denote by $L^2(V,\pi)$ the Hilbert space of $L^2$ summable functions on $V,$ equipped with the inner product
$$\IP{ f ,g } = \IP{ f,g}_\pi : = \sum_x\pi(x) f(x) \overline{ g(x)  }.$$
It is well-known, the Laplacian $\Delta$ is a bounded self-adjoint operator on $L^2(V,\pi),$ whose spectrum is contained in $[0,2].$ We write $\lambda_{\min}$ for the {\em bottom} of the spectrum of $\Delta.$

The Laplacian with Dirichlet boundary condition on finite subsets of networks has been investigated in the literature, see e.g. \cite{Dod84,Frie93,CouGri98,ChungYau00,BHJ14}.
For finite $\Omega \subset V,$ the Laplacian with Dirichlet boundary conditions on $\Omega,$ denote by $\Delta_\Omega,$ is defined as
the Laplacian $\Delta$ restricted to the subspace $$L^2(\Omega) := \{ f \in L^2(V,\pi) \ : \ f\big|_{G \setminus \Omega} \equiv 0 \}.$$ The eigenvalues of $\Delta_\Omega,$ called  {\em Dirichlet eigenvalues} on $\Omega$, are ordered by
$$0<\lambda_1\leq\lambda_2\leq\cdots\leq\lambda_{|\Omega|},$$ where $|\cdot|$ denotes the cardinality of the subset. 
We are interested in proving universal inequalities on graphs, in particular Yang-type inequalities \eqref{eq:Yangint1} and \eqref{eq:hyp1}. Due to the discrete nature of graphs, some modification is required.

\begin{dfn}
We say that the network $(V,c)$ satisfies \define{Yang's inequality} (resp. the \define{Yang-type inequality})
with constant $C_Y$ (resp. $C_{YT}$) if the following holds
for any finite subset $\Omega \subset G$:

Let $0 <\lambda_1 \leq \lambda_2 \leq \cdots \leq \lambda_{|\Omega|}$ be the Dirichlet eigenvalues 
of $\Omega$.
Then, for any $k < |\Omega|$,
$$ \sum_{i=1}^{k} | \lambda_{k+1} - \lambda_i |^2 
\leq C_Y \cdot \sum_{i=1}^k (\lambda_{k+1} - \lambda_i)  (\lambda_i - \lambda_{\min} ) . $$
\begin{equation}\label{eq:Yangtype}(\mathrm{resp.}\quad \sum_{i=1}^{k} | \lambda_{k+1} - \lambda_i |^2 (1- \lambda_i ) 
\leq C_{YT} \cdot \sum_{i=1}^k (\lambda_{k+1} - \lambda_i)  (\lambda_i - \lambda_{\min} ) . )\end{equation}
\end{dfn}

Since $\lambda_i\leq 2,$ for any $i\geq 1,$ one easily sees that in case of $\lambda_{\min}=0,$ the Yang-type inequality implies Yang's inequality with $C_Y=C_{YT}+2.$
Following the arguments in \cite{Yang91,Ash99,CY07}, the first author {\em et al.\ } \cite{HLS17} proved that the integer lattice $\Z^n,$ a discrete analog of $\R^n,$ satisfies Yang-type inequality,
with constant $C_{YT}=\frac{4}{n}.$ Recently, Kobayashi \cite{Kobayashi20} proved the Yang-type inequality for the eigenvalues of the Laplacian (not Dirichlet eigenvalues) of a finite edge-transitive graph. 

Note that $\Z^n$ can be regarded as a Cayley graph of a free Abelian group. In this paper, we prove Yang-type inequalities for more general Cayley graphs of finitely generated infinite groups.

\subsection{Amenable groups}

Our first result regards amenable groups.
Let $G$ be a finitely generated amenable group.
Consider some probability measure $\mu$ on $G$ (which we think of as a non-negative function 
$\mu :G \to [0,1]$ such that $\sum_x \mu(x) = 1$).  
Assume that $\mu$ is symmetric, \ie $\mu(x) = \mu(x^{-1})$ for all $x \in G$.
Then $\mu$ induces a corresponding Cayley graph (or network) by setting the conductances 
$c(x,y) = \mu(x^{-1} y)$.  This network corresponds to the $\mu$-random walk on $G$.
This network is denoted by $(G,\mu)$.

\begin{thm} \label{thm:amenable}
Let $G$ be a finitely generated infinite amenable group.
Let $\mu$ be a symmetric probability measure on $G$, and consider the Cayley network $(G,\mu)$
of $G$ with respect to $\mu$.
Set $ \mu_* : = \inf_{1 \neq y \in \supp(\mu) } \mu(y)$.

Then, the network $(G,\mu)$ satisfies Yang's inequality,
with constant $C_Y = \frac{6}{ \mu_*}$. 
\end{thm}

For finitely generated groups with Abelian quotients, \ie those groups which admit homomorphisms onto $\Z^n$ for some $n,$ we prove the Yang-type inequality with $C_{YT}=\frac{4}{n}$ for specific $\mu$-random walks, see Theorem~\ref{thm:Zd}. This extends the result for $\Z^n$ from \cite{HLS17}.

\subsection{Free groups}

Next, we consider Yang-type inequalities on regular trees, which can be regarded as Cayley graphs of free groups. Let $\T_d,$ $d\geq3,$ be a $d$-regular tree with the conductances of the edges $c(x,y) = \1{x \sim y} \tfrac1d$, which is a discrete analog of hyperbolic space $\mathbb{H}^d$. The Laplacian corresponds to the generator of the {\em simple random walk} on $\T_d$. 
As is well-known, the bottom of the spectrum of $\T_d$ is $1-\frac{2\sqrt{d-1}}{d}$. 
Following the arguments in \cite{ChengYang09}, we prove the following result.

\begin{thm} \label{thm:trees}
The network given by the simple random walk on the $d$-regular tree $\T_d$ (where $d>2$)
satisfies the Yang-type inequality with constant $C_{YT}=\frac{8\sqrt{d-1}}{d}$.
\end{thm}
We sketch the proof strategies of Theorem~\ref{thm:amenable} and Theorem~\ref{thm:trees}: 
By the variational principle, for an upper bound estimate of eigenvalues, 
it suffices to construct appropriate test functions. 
Following the arguments in \cite{Yang91,ChengYang06JPS}, for any network and any test function $\alpha:V\to\R,$ we prove the Dirichlet eigenvalues satisfy some crucial estimate involving $\alpha,$ see Lemma~\ref{lem:main bound}, a discrete analog of \cite[Proposition~1]{ChengYang06JPS}. This enables us to derive the Yang-type inequality with choice of $\alpha$ with nice properties for $\Delta \alpha$ and the gradient of $\alpha.$ For $\R^n$ or $\Z^n,$ as in \cite{Yang91,CY07,HLS17}, linear functions are good candidates for test functions. 

In order to generalize the result to Cayley graphs of amenable groups, \ie Theorem~\ref{thm:amenable}, we use {\em harmonic cocycles} as test functions. The existence of harmonic cocycles for amenable groups was proved by \cite{mok1995harmonic, korevaari1997global}. 

For $\mathbb{H}^n,$ Cheng and Yang \cite{ChengYang09} used Busemann functions of geodesic rays to prove Yang-type inequality \eqref{eq:hyp1}. To extend the result to $\T_d,$ \ie Theorem~\ref{thm:trees}, we use the discrete analogs of Busemann functions as test functions.

The paper is organized as follows:
In next section, we introduce some basic facts on networks.
In Section~\ref{scn:main bound}, we prove the useful estimate of eigenvalues for general networks, Lemma~\ref{lem:main bound}.
Section~\ref{sec:app} is devoted to the proofs of main results, Theorem~\ref{thm:amenable} and Theorem~\ref{thm:trees}. In the last section, we derive some applications of the Yang-type inequality, such as the Paley-Polya-Weinberger inequality and the Hile-Protter inequality, etc.

\section{Notation and basic operators}

\subsection{$\Gamma$ calculus}

Let $(V,c)$ be a network on the set of vertices $V$ with the conductance $c$.
We allow $c(x,x)>0$,  which corresponds to a self-edge at $x \in V$.

Recall the inner product on functions defined in the introduction
$$ \IP{ f,g } = \sum_x \pi(x) f(x) \overline { g(x) } . $$
Accordingly we write
$|| f ||^2 = || f ||_\pi^2 : = \IP{ f,f}$, 
and the space of $L^2$ summable functions is given by 
$L^2(V,\pi) : = \{ f :V \to \C \ : \ || f || < \infty \}$.

The \define{Dirichlet energy} is defined to be 
$$ \Ee(f,g) := %\IP{ \nabla f , \nabla g } =  
\sum_{x ,y } c(x,y) (f(x) - f(y)) \overline{ (g(x) - g(y)) } , $$
and $\Ee(f) : = \Ee(f,f)$.
If $f ,g \in L^2(V,\pi),$ then it is not difficult to prove the  ``integration by parts'' formula,
$$ \Ee(f,g) = 2 \IP{ \Delta f , g } = 2 \IP{ f , \Delta g } . $$

Define the so called {\em carr\'e du champ} operator (at $x \in V)$ as follows:
$$ 2 \Gamma(f,g)(x) := \big(  f \Delta \bar g + \bar g \Delta f - \Delta (f \bar g) \big) (x) 
\sum_y P(x,y) (f(x) - f(y)) \overline{ (g(x) - g(y)) } ,$$
and $\Gamma(f) := \Gamma(f,f)$.
Note that $\Gamma$ is symmetric and bi-linear.

Finally we define the scalar-valued (non-linear) functional:
$$ \Lambda(f,g) = \tfrac14 \sum_{x,y} c(x,y) |f(x)-f(y)|^2 \cdot |g(x) - g(y)|^2 
. $$

\subsection{Identities}

In this section we summarize a few identities which we will require in the analysis below.
All are straightforward and easy to prove, and hold for all $f,g \in L^2(V,\pi)$.

\begin{align}
\Ee(f,g) & = 2 \sum_x \pi(x) \Gamma(f,g)(x) = 2 \IP{ \Gamma(f,g) , 1 } .
\end{align}

Also, note that
\begin{align*}
\IP{ \Gamma(f,g) , g }  & = \tfrac12 \sum_{x,y} P(x,y) (f(x) - f(y)) \overline{ (g(x)-g(y)) g(x) } \pi(x) 
\end{align*}
Since $\pi(x) P(x,y) = c(x,y) = \pi(y) P(y,x)$,
\begin{align*}
\Ee( f, g^2) & 
= \sum_{x,y} c(x,y) (f(x)-f(y)) \overline{ (g(x)^2 - g(y)^2) } \\
& = \sum_{x,y} P(x,y) (f(x) - f(y)) \overline{ (g(x)-g(y)) g(x) } \pi(x) 
\\
& \qquad \qquad  + \sum_{x,y} P(x,y) (f(x)-f(y)) \overline{ (g(x)-g(y)) g(y) } \pi(x) \\
& = 4 \IP{ \Gamma(f,g) , g } .
\end{align*}
So in conclusion
\begin{align}
\IP{ 2 \Gamma(f,g) , g } & = \IP{ \Delta f , g^2 } . 
\end{align}

We also may compute, 
\begin{align*}
\IP{ 2\Gamma(f,g) , f \cdot g } & = \sum_{x,y} c(x,y) ( f(x) - f(y)) (g(x) - g(y)) f(x) g(x) \\
& = \sum_{x,y} c(x,y) ( f(x) - f(y)) (g(x) - g(y)) \cdot \tfrac{f(x) g(x) + f(y) g(y) }{2} \\
& = \sum_{x,y} c(x,y) ( f(x) - f(y)) (g(x) - g(y)) \cdot 
\tfrac{(f(x) + f(y)) (g(x) + g(y)) + (f(x)-f(y)) (g(x) - g(y))}{4} \\
& = \tfrac14 \sum_{x,y} c(x,y) ( f(x)^2 - f(y)^2) (g(x)^2 - g(y)^2)
\\
& \qquad \qquad + \tfrac14 \sum_{x,y} c(x,y) | f(x) - f(y) |^2 \cdot |g(x) - g(y) |^2 ,
\end{align*}
which culminates in
\begin{align} \label{eqn:Gamma and Lambda}
\IP{ 2\Gamma(f,g) , f \cdot g } & = \tfrac14 \Ee(f^2 , g^2) + \Lambda(f,g) .
\end{align}

\section{Universal inequality}
\label{scn:main bound}

The following is an analogue of \cite[Proposition~1]{ChengYang06JPS}. It is the main estimate which will imply our results.

Let $(V,c)$ be a network. 
Let $\Omega \subset V$ be a finite subset of size $n = |\Omega|$.
Let $u_1, \ldots,u_n$ be an orthonormal basis of eigenvectors for $\Delta_\Omega$ defined on 
the subspace $L^2(\Omega)$ 
of $L^2(V,\pi);$ 
that is,
\begin{itemize}
\item $\lambda_{\min} \leq \lambda_1 \leq \lambda_2 \leq \cdots \leq \lambda_n$, 
\item $\Delta u_i = \lambda_i u_i$,
\item $u_i \big|_{G \setminus \Omega} \equiv 0$,
\item $\IP{ u_i , u_j} = \1{i=j}$.
\end{itemize}
Since the Laplacian is self-adjoint, such an orthonormal basis exists, 
$\lambda_i \in \R$ and $u_i$ are real valued.

We call such a collection $(\lambda_i , u_i)_{i=1}^n$ the \define{Dirichlet system} for $\Omega$.

\begin{lem}
\label{lem:main bound}
Let $(V,c)$ be a network. Let $\Omega \subset V$ be a finite subset of size $n = |\Omega|$.
Let $(\lambda_i, u_i)_{i=1}^n$ be the Dirichlet system for $\Omega$.

Then,
for any $k<n$ and any $\alpha:V \to \R$ we have  
$$ \sum_{i=1}^k | \lambda_{k+1} - \lambda_i |^2 \Big( \IP{ \Gamma(\alpha) , u_i^2 } 
- \Lambda(\alpha,u_i) \Big)
\leq \sum_{i=1}^k (\lambda_{k+1}-\lambda_i)  || u_i \cdot \Delta \alpha - 2 \Gamma(\alpha,u_i) ||^2 . $$
\end{lem}

\begin{proof}
Let $\alpha :G \to \R$.
Fix some $1 \leq k < n$.
Set
\begin{align*}
a_{ij} & = \IP{ u_i \cdot \alpha , u_j } , \\ %= \sum_x \pi(x) u_i(x) u_j(x) \cdot \alpha(x) , \\
\vphi_i & = u_i \cdot \alpha - \sum_{j=1}^k a_{ij} \cdot u_j  , \\
\alpha_i & =  u_i \cdot \Delta \alpha - 2 \Gamma(u_i, \alpha) , \\
b_{ij} & = \IP{ \alpha_i , u_j } , \\
w_i & = \IP{ \alpha_i , \vphi_i } , \\
z_i & = \IP{ \alpha_i , u_i \cdot \alpha }  , \\
y_i & = \Lambda(\alpha,u_i) .
\end{align*}

We collect a few observations regarding these quantities:

For all $1 \leq i,j \leq k$,
\begin{align} \label{eqn:vphi orthogonal}
 \IP{ \vphi_i , u_j} & = \IP{ u_i \cdot \alpha , u_j } - \sum_{\ell=1}^k \IP{ u_\ell , u_j } a_{i\ell}
 = a_{ij} - a_{ij} = 0 .
\end{align}

Also, $a_{ij} = a_{ji}$ and 
since the Laplacian is self-adjoint,
\begin{align*}
\lambda_j \cdot a_{ij} & = \IP{  u_i \cdot \alpha , \Delta u_j } = \IP{ \Delta (u_i \cdot \alpha) , u_j } \\
& = \IP{ \Delta u_i \cdot \alpha + u_i \cdot \Delta \alpha - 2 \Gamma(u_i , \alpha) , u_j } \\
& = \lambda_i \cdot a_{ij} + \IP{ \alpha_i , u_j } = \lambda_i \cdot a_{ij} + b_{ij} ,
\end{align*}
which proves that 
for all $1 \leq i,j \leq k$,
\begin{align} \label{eqn:bij}
b_{ij} & = - b_{ji} = (\lambda_j - \lambda_i ) \cdot a_{ij}
\end{align}

\begin{align} \label{eqn:Lap vphi}
\Delta \vphi_i & = \Delta ( u_i \cdot \alpha) - \sum_{j=1}^k \Delta u_j \cdot a_{ij} 
= \lambda_i u_i \cdot \alpha + \alpha_i - \sum_{j=1}^k \lambda_j u_j \cdot a_{ij} .
\end{align}

Since $\IP{ u_i , u_j } = \1{i=j}$,
\begin{align} \label{eqn:norm alphai}
||  \alpha_i - \sum_{j=1}^k  b_{ij} \cdot u_j ||^2  & = 
|| \alpha_i ||^2 + \sum_{j=1}^k || b_{ij} \cdot u_j ||^2 - 2 \sum_{j=1}^k b_{ij} \cdot \IP{ \alpha_i , u_j } 
\nonumber \\
& 
= || \alpha_i ||^2 - \sum_{j=1}^k |b_{ij} |^2 .
\end{align}

By \eqref{eqn:bij} we know that $-\IP{ \alpha_i , u_j } = -b_{ij} = (\lambda_i - \lambda_j) a_{ij}$, so
\begin{align} \label{eqn:wizi}
w_i & = z_i - \sum_{j=1}^k \IP{ \alpha_i , a_{ij} \cdot u_j} = z_i 
+ \sum_{j=1}^k (\lambda_i - \lambda_j) |a_{ij}|^2 .
\end{align}

By \eqref{eqn:Gamma and Lambda} we have that
\begin{align*}
\IP{2 \Gamma(u_i,\alpha) , u_i \cdot \alpha } & 
= \tfrac12 \IP{ \Delta (\alpha^2) , u_i^2 } + \Lambda(\alpha,u_i) .
\end{align*}
Thus,
\begin{align} \label{eqn:ziyi}
z_i + y_i & = \IP{ u_i \cdot \Delta \alpha - 2 \Gamma(u_i,\alpha) , u_i \cdot \alpha } 
+ \Lambda(\alpha,u_i) \nonumber \\
& = \IP{ u_i \cdot \Delta \alpha , u_i \cdot \alpha } - \tfrac12 \IP{ \Delta (\alpha^2) , u_i^2 } \nonumber \\
& = \IP{ \Delta \alpha \cdot \alpha - \tfrac12 \Delta(\alpha^2) , u_i^2 } 
= \IP{ \Gamma(\alpha) , u_i^2} .
\end{align}

%%%
%%%

By \eqref{eqn:vphi orthogonal} 
we get that $\IP{ \vphi_i , u_i \cdot \alpha } = || \vphi_i ||^2$.
Also, since $\vphi_i$ is orthogonal to $\{ u_1 , \ldots, u_k \}$, using \eqref{eqn:Lap vphi},
\begin{align*}
\lambda_{k+1} || \vphi_i ||^2 & \leq \IP{ \Delta \vphi_i , \vphi_i } \\
& = \IP{ \lambda_i u_i \cdot \alpha + \alpha_i - \sum_{j=1}^k \lambda_j u_j \cdot a_{ij}  , \vphi_i } \\
& = w_i + \lambda_i \IP{ u_i \cdot \alpha , \vphi_i } = w_i + \lambda_i || \vphi_i ||^2 .
\end{align*}
Using the Cauchy-Schwarz inequality and \eqref{eqn:norm alphai},
\begin{align*}
(\lambda_{k+1} - \lambda_i) |w_i|^2 & = (\lambda_{k+1} - \lambda_i) \Big| \langle \alpha_i 
- \sum_{j=1}^k b_{ij} \cdot u_j , \vphi_i \rangle \Big|^2 \\
& \leq (\lambda_{k+1}- \lambda_i) || \vphi_i ||^2 \cdot \Big( || \alpha_i ||^2 - \sum_{j=1}^k |b_{ij}|^2 \Big) \\
& \leq w_i \cdot \Big( || \alpha_i ||^2 - \sum_{j=1}^k |b_{ij}|^2 \Big) .
\end{align*}
Thus,
\begin{align}
\label{eqn:lambda wi bound}
(\lambda_{k+1} - \lambda_i) w_i & \leq || \alpha_i ||^2 
- \sum_{j=1}^k |\lambda_i - \lambda_j |^2 \cdot |a_{ij}|^2  . 
\end{align}
By \eqref{eqn:wizi},
\begin{align*}
\sum_{i=1}^k & |\lambda_{k+1}- \lambda_i |^2 w_i 
= \sum_{i=1}^k |\lambda_{k+1}-\lambda_i |^2 z_i
+ \sum_{i,j=1}^k |\lambda_{k+1}-\lambda_i|^2 (\lambda_i-\lambda_j) |a_{ij}|^2 \\
& = \sum_{i=1}^k |\lambda_{k+1}-\lambda_i |^2 z_i
+ \tfrac12 \sum_{i,j=1}^k \big( | \lambda_{k+1} - \lambda_i|^2 - |\lambda_{k+1}-\lambda_j|^2 \big) 
(\lambda_i-\lambda_j) |a_{ij}|^2 \\
& = \sum_{i=1}^k |\lambda_{k+1}-\lambda_i |^2 z_i
-\sum_{i,j=1}^k \big( \lambda_{k+1} - \tfrac{\lambda_i + \lambda_j }{2} \big) 
|\lambda_i-\lambda_j|^2 |a_{ij}|^2 \\
& = \sum_{i=1}^k |\lambda_{k+1}-\lambda_i |^2 z_i
- \sum_{i,j=1}^k ( \lambda_{k+1} - \lambda_i )
|\lambda_i-\lambda_j|^2 |a_{ij}|^2 .
\\
\end{align*}
Multiplying \eqref{eqn:lambda wi bound} by $\lambda_{k+1}-\lambda_i$ and summing over $i$,
we obtain
\begin{align} \label{eqn:zi bound}
\sum_{i=1}^k | \lambda_{k+1}- \lambda_i|^2 z_i & \leq \sum_{i=1}^k (\lambda_{k+1} - \lambda_i ) 
|| \alpha_i ||^2 .
\end{align}
The proof is now complete using $z_i = \IP{ \Gamma(\alpha) , u_i^2} - \Lambda(\alpha,u_i)$ 
by \eqref{eqn:ziyi}.
\end{proof}

Let $\H$ be a Hilbert space and $\alpha : V \to \H$.
We extend the definitions of the inner product and of $\Gamma , \Lambda$ by defining
\begin{align*}
2 \Gamma(\alpha,u) & = \sum_y P(x,y) (u(x) - u(y) ) \cdot (\alpha(x) - \alpha(y)) , \\
2 \Gamma(\alpha)(x) & = \sum_y P(x,y) || \alpha(x) - \alpha(y) ||^2_\H , \\
\IP{ \alpha , u } & = \sum_x \pi(x) u(x) \cdot \alpha(x) , \\
|| \alpha ||^2 & = \IP{ \alpha , \alpha } = \sum_x \pi(x) || \alpha(x) ||_\H^2 , \\
\Lambda(\alpha,u) & = \tfrac14 \sum_{x,y} c(x,y) |u(x) - u(y)|^2 \cdot || \alpha(x) - \alpha(y) ||^2_\H 
\end{align*}
Here $u:V \to \R$ is any (finitely supported) real valued function.
With this notation, we have the following theorem generalizing Lemma \ref{lem:main bound}.

\begin{thm}
\label{thm:main bound Hilbert}
Let $(V,c)$ be a network. Let $\Omega \subset V$ be a finite subset of size $n = |\Omega|$.
Let $(\lambda_i, u_i)_{i=1}^n$ be the Dirichlet system for $\Omega$.
Let $\H$ be a Hilbert space and let $\alpha : V \to \H$.

Then for any $k<n$,
$$ \sum_{i=1}^k | \lambda_{k+1} - \lambda_i |^2 \cdot 
\Big( \IP{ \Gamma(\alpha) , u_i^2 } - \Lambda(\alpha,u_i) \Big)
\leq \sum_{i=1}^k (\lambda_{k+1}- \lambda_i) \cdot 
|| u_i \cdot \Delta \alpha - 2 \Gamma(\alpha,u_i) ||^2 . $$
\end{thm}

Note that when $\H=\R$ this is exactly Lemma \ref{lem:main bound}.

\begin{proof}
Let $h \in \H$ be any non-zero vector.
Define the function $\alpha' : V \to \R$ by $\alpha'(x) = \IP{ \alpha(x) , h }_\H$.
Plugging this into Lemma \ref{lem:main bound} we see that we only need to compute
$\Gamma(\alpha') , \Lambda(\alpha',u_i) , \Gamma(\alpha',u_i) , \Delta \alpha'$.
It is simple to verify that
\begin{align*}
\Delta \alpha' & = \IP{ \Delta \alpha , h }_\H , \\
\Lambda(\alpha',u_i) & = \tfrac14 \sum_{x,y} c(x,y) |u_i(x) - u_i(y)|^2 \cdot 
| \IP{ \alpha(x) - \alpha(y) , h }_\H |^2 , \\
2\Gamma(\alpha')(x) & = \sum_y P(x,y) | \IP{ \alpha(x) - \alpha(y), h }_\H |^2 , \\
2 \Gamma(\alpha' , u_i)(x) & = \sum_y P(x,y)  (u_i(x) - u_i(y) ) \cdot \IP{ \alpha(x) - \alpha(y) , h }_\H .
\end{align*}
Summing this over $h$ in an orthonormal basis for $\H$, we have the theorem.
\end{proof}

\section{The proof of main results}\label{sec:app}
\subsection{Amenable groups}

One application of Theorem \ref{thm:main bound Hilbert} is for the case of amenable groups.
Given a finitely generated group, there is a natural network one may define.
Actually, the initial data is a finitely generated group $G$ and a probability measure $\mu$
on $G$, which is assumed to be {\em symmetric}, \ie $\mu(x) = \mu(x^{-1})$.
This measure is used to construct the {\em random walk} on $G$, 
which is just the Markov chain with transition matrix $P(x,y) = \mu(x^{-1} y)$.
This Markov chain is precisely the reversible Markov chain associated to the network
on $G$ given by conductances $c(x,y) = \mu(x^{-1} y)$.
We denote this network by $(G,\mu)$, and call it the \define{Cayley network} of $G$ with respect to $\mu$.
(Since $\mu$ is a probability measure, in this case $\pi(x) = 1$ for all $x$.)

For a probability measure $\mu$ on $G$, define
$$ \mu_* : = \inf_{1 \neq y \in \supp(\mu) } \mu(y) . $$
Note that $\mu$ has finite support if and only if $\mu_*>0$.

Recall that 
Kesten's amenability criterion \cite{kesten1959full} states 
that the bottom of the spectrum of $\Delta$ is $0$ if and only if
$G$ is an amenable group.

We are now ready to prove Theorem~\ref{thm:amenable}.
\begin{proof}[Proof of Theorem~\ref{thm:amenable}]
Since $G$ is amenable and infinite, it does not have Kazhdan property (T).
(This is very well known, and an easy exercise following the definitions of property (T) and amenability.
See \eg \cite[Chapter 7]{Pete}.)
It follows from \cite{mok1995harmonic, korevaari1997global} that
there exists a Hilbert space $\H$ on which the group $G$ acts by unitary operators,
with a {\em harmonic cocycle} $\alpha : G \to \H$. 
That is, $\alpha(xy) = \alpha(x) + x. \alpha(y)$ for all $x,y \in G$ and 
$\Delta \alpha \equiv 0$.    (For a short proof see \eg \cite{ozawa2018functional}.)

Since the $G$-action is unitary, we may compute that
$$ || \alpha(x) - \alpha(xy) ||^2_\H = || \alpha(y) ||_\H^2 , $$
so
$$ 2\Gamma(\alpha)(x)  = \sum_y \mu(y) || \alpha(y) ||_\H^2 , $$
is a constant function.

Now, if $u$ is an eigenfunction of unit length, with $\Delta u = \lambda u$, then
$$ \IP{ \Gamma(\alpha) , u^2 } = \Gamma(\alpha) \cdot  \sum_x \pi(x) u(x)^2 = \Gamma(\alpha) . $$
Also,
\begin{align*}
4 \Lambda(\alpha,u) & =  \sum_{x,y} c(x,y) | u(x) - u(y) |^2 \cdot 
|| \alpha(x) - \alpha(y)  ||_\H^2
\\
& =  \sum_{x,y} \mu(y) | u(x) - u(xy) |^2 \cdot || \alpha(y) ||_\H^2 . %\\
\end{align*}
since for any $1 \neq y \in \supp(\mu)$,
$$ || \alpha(y) ||_\H^2 \leq \frac{1}{ \mu_* } \sum_y\mu(y) || \alpha(y) ||_\H^2 
\leq \frac{1}{\mu_* } \cdot 2\Gamma(\alpha) , $$
we get that
\begin{align*}
4 \Lambda(\alpha,u) & \leq 
\frac{1}{ \mu_* } \cdot 2 \Gamma(\alpha) \cdot \sum_{x,y} \mu(y) |u(x) - u(xy) |^2
= \frac{4}{\mu_*} \Gamma(\alpha) \cdot \lambda .
\end{align*}
Finally,
\begin{align*}
2 \Gamma(\alpha,u)(x) & = \sum_y \mu(y) (u(x)- u(xy) ) \cdot (\alpha(x) - \alpha(xy) )
= - \sum_y \mu(y) (u(x) - u(xy)) \cdot x.\alpha(y) .
\end{align*}
Since $G$ acts unitarily on $\H$, we have by Jensen's inequality,
\begin{align*}
|| 2 \Gamma (\alpha,u) ||^2 & = \sum_x || \sum_y \mu(y) (u(x)-u(xy)) \cdot \alpha(y) ||_\H^2 
\\
& \leq \sum_{x,y} \mu(y) |u(x)-u(xy)|^2 \cdot || \alpha(y)||^2_\H = 4 \Lambda(\alpha,u) .
\end{align*}
Plugging all the above into Theorem \ref{thm:main bound Hilbert} we arrive at
\begin{align*}
\sum_{i=1}^k | \lambda_{k+1}-\lambda_i|^2 \cdot \Gamma(\alpha)  & \leq
\sum_{i=1}^k (\lambda_{k+1} - \lambda_i ) \cdot \Lambda(\alpha,u_i) \cdot ( 4 + \lambda_{k+1}- \lambda_i ) \\
& \leq \sum_{i=1}^k (\lambda_{k+1} - \lambda_i ) \lambda_i \cdot 
\frac{6}{\mu_*} \cdot \Gamma(\alpha) ,
\end{align*}
where we have used that $\lambda_{k+1} - \lambda_i \leq 2$.
This completes the proof.
\end{proof}

\subsection{Groups with Abelian quotients}

For general groups with Abelian quotients, we can prove the Yang-type inequality, analogous to the result in \cite{HLS17}. 

\begin{thm} \label{thm:Zd}
Let $G$ be a finitely generated group.
Let $\alpha : G \to \Z^n$ be a surjective homomorphism.
Let $S  = \{ s_1, \ldots, s_n, k_1, \ldots, k_m \}$ be a generating set for $G$
so that $(\alpha(s_j))_{j=1}^n$ is the standard basis of $\Z^n$, and such that $\alpha(k_j)=0$
for all $j=1,\ldots,m$.
Let $\mu$ be a symmetric measure supported on $S \cup S^{-1}$.
Let $\eps = 1- \sum_{j=1}^n (\mu(s_j) + \mu(s_j^{-1}) )$.
(\eg one may take $\mu(k_j) = \mu(k_j^{-1}) = \tfrac{\eps}{2n}$ and $\mu(s_j) = \mu(s_j^{-1})=\tfrac{1-\eps}{2n}$.)

Then, the network $(G,\mu)$ satisfies the following: For any finite $\Omega\subset G$ and $k<|\Omega|,$ 
$$ \sum_{i=1}^k | \lambda_{k+1} - \lambda_i |^2 \cdot (1-\eps- \lambda_i)
\leq  8 \max_j \mu(s_j)  \cdot  \sum_{i=1}^k (\lambda_{k+1}- \lambda_i) \cdot \lambda_i . $$
\end{thm}

\begin{rem}
When we choose $\mu(k_j) = \mu(k_j^{-1}) = \tfrac{\eps}{2n}$ and $\mu(s_j) = \mu(s_j^{-1})=\tfrac{1-\eps}{2n}$,
we get the Yang-type inequality up to an $\eps$-defect, with constant at most $\frac{4}{n}$.
\end{rem}

\begin{rem}
The case $G \cong \Z^n$ was already treated in \cite{HLS17}, where the same result was shown,
using similar methods. This is the case $\eps=0$ and $\mu(s_j) = \mu(s_j^{-1}) = \frac{1}{2 n}$ 
in the above theorem.
\end{rem}

\begin{proof}
The main advantage of $\alpha$ being a homomorphism is that 
$$ \mu(y) \alpha(y) = 
\begin{cases}
\pm \mu(s_j) e_j & y = (s_j)^{\pm 1} , \\
0, & \textrm{ otherwise. } 
\end{cases},
$$ 
where $\{e_j\}_{j=1}^n$ is the standard basis of $\Z^n.$ 
Thus, for the Euclidean Hilbert space $\H = \R^n$,
$$ 2 \Gamma(\alpha)(x) = \sum_y \mu(y) || \alpha(x) - \alpha(xy) ||_\H^2 
= \sum_{j=1}^n (\mu(s_j) + \mu(s_j^{-1}) ) = 1-\eps , $$
for any $x \in G$.
Also, $\Delta \alpha \equiv 0$.
Now, 
if $u$ is an eigenfunction of unit length, with $\Delta u = \lambda u$, then
$$ \IP{ \Gamma(\alpha) , u^2 } = \Gamma(\alpha)= \tfrac12(1-\eps) . $$
We may bound
\begin{align*}
4 \Lambda(\alpha,u) &  = \sum_{x,y} \mu(y) |u(x) - u(xy) |^2 \cdot ||\alpha(y) ||_\H^2 \\
& = \sum_x \sum_{j=1}^n \mu(s_j) \Big( |u(x) - u(xs_j) |^2 + |u(x) - u(xs_j^{-1}) |^2 \Big) \\
& \leq \sum_{x,y } \mu(y) |u(x) - u(xy)|^2  = 2 \lambda .
\end{align*}

As in the proof of Theorem \ref{thm:amenable}, 
\begin{align*}
2 \Gamma(\alpha,u)(x)  & = \sum_{j=1}^n \mu(s_j) (u(x) - u(xs_j) - u(x) + u(x s_j^{-1} ) ) \cdot \alpha(s_j) , \\
|| 2 \Gamma(\alpha,u) ||^2 & = \sum_x \sum_{j=1}^n \mu(s_j)^2 | u(xs_j^{-1}) - u(xs_j) |^2\\
&\leq 2\sum_x \sum_{j=1}^n \mu(s_j)^2 (| u(x) - u(xs_j) |^2+ | u(x) - u(xs_j^{-1}) |^2)
\\
& \leq 2\max_j \mu(s_j) \sum_{x,y} \mu(y) |u(x) - u(xy) |^2 = \max_j \mu(s_j) \cdot 4 \lambda . 
\end{align*}

Plugging all of this into Theorem \ref{thm:main bound Hilbert}, we arrive at
$$ \sum_{i=1}^k | \lambda_{k+1} - \lambda_i |^2 \cdot (1-\eps-\lambda_i )
\leq  8 \max_j \mu(s_j)  \cdot  \sum_{i=1}^k (\lambda_{k+1}- \lambda_i) \cdot \lambda_i . $$
\end{proof}

\subsection{Trees}

In this section, we prove the Yang-type inequality for $d$-regular tree $\T_d,$ $d\geq 3,$ with the conductances of the edges $c(x,y) = \1{x \sim y} \tfrac1d.$

\begin{proof}[Proof of Theorem~\ref{thm:trees}]
Fix a ray to infinity, and an origin $o$.
Let $b$ be the {\em Buseman function} corresponding to the ray with $b(o)=0$.
That is: let $o=x_0 \sim x_1 \sim \cdots \sim x_n \sim x_{n+1} \sim \cdots$ be an infinite
{\em simple} path, so $x_i \neq x_j$ for all $i \neq j$.
Becase $\T_d$ is a tree, this path is necessarily a geodesic: the distance between $x_j, x_i$ 
in the graph is always $|j-i|$. This path is the {\em ray} mentioned above.
Now, for any $j \geq 0$ set $b(x_j) : = -j$.
Furthermore, 
for any vertex $z$, let $z_*$ be the closest vertex to $z$ from the above path.
Set $b(z) = b(z_*) + \dist(z,z_*)$.

The important properties of $b$ are thus: 
$b:\T_d \to \Z$ is a function such that $b(o) = 0$ and such that
every vertex $x$ has $d-1$ neighbors $y\sim x$ with $b(y) = b(x)+1$,
and exactly one neighbor $\vec x \sim x$ with $b(\vec x) = b(x) -1$. One easily sees that
$$ 2\Gamma(b)(x)= 1 \qquad \forall x \in \T_d . $$

It is also simple to  check that the function $f(x) = (\tfrac{\xi}{\sqrt{ d-1} })^{b(x)}$ satisfies
$$ \Delta f(x) = f(x) \cdot \big( 1  - \tfrac{ \sqrt{ d-1} }{d} \cdot (\xi + \xi^{-1} ) \big) . $$
Hence, if $\lambda = 1 - \tfrac{2 \sqrt{d-1} }{d}$ (which corresponds to choosing $\xi=1$, 
maximizing the above expression) 
then $\Delta f = \lambda f$.
Coincidentally, this is the bottom of the $L^2$ spectrum of $\Delta$,
\ie $\lambda_{\min} = 1 - \tfrac{2 \sqrt{d-1} }{d}$.

For any $x$ let $\vec x$ be the unique vertex with $b(\vec x) = b(x) - 1$.
For a function $f$ let $\vec{f}(x) : = f(\vec x)$.
Note that as $x$ ranges over the whole graph, the pair $(x,\vec x)$ ranges over
all edges in the graph, each edge counted exactly once in the direction of decreasing the Buseman 
function $b$.
Thus,
\begin{align*}
|| f - \vec f ||^2 & = \sum_x | f(x) - f ( \vec x) |^2 =   
\tfrac12 \sum_{x \sim y} |f(x) - f(y)|^2 \\
& = \tfrac{d}2 \sum_{x,y} c(x,y) | f(x) - f(y) |^2 = d \iP{ \Delta f , f } .
\end{align*}
Also, the map $x \mapsto \vec x$ is a $(d-1)$-to-$1$ map.  So,
\begin{align} \label{eqn:f and vec f 1}
|| \vec f ||^2 &  = \sum_x  |f(\vec x)|^2 = \sum_y \sum_{x \ : \ \vec x = y } |f(y)|^2
= (d-1) || f ||^2 .
\end{align}
Thus,
\begin{align} \label{eqn:f and vec f 2}
d \iP{ \Delta f , f } & = || f - \vec f ||^2 = d \cdot || f ||^2 - 2 \iP{ f, \vec f } . 
\end{align}

Note that the Buseman function satisfies:
$$ \Delta b (x) = \sum_y P(x,y) (b(x) - b(y)) = - \tfrac{d-2}{d} = : -\gamma , $$
and also $|b(x) - b(y) |=1$ for any $x \sim y$.

Let $u$ be an eigenfunction $\Delta u = \lambda u$.
Note that 
$$ \IP{ 2 \Gamma(b,u) , u } = \tfrac12 \Ee( b,u^2) = \IP{ \Delta b ,u^2 } = -\gamma || u ||^2 . $$
Thus,
\begin{align} \label{eqn:Gamma(b,u) 1}
|| 2\Gamma(b,u) & - u \Delta b ||^2  = 4 || \Gamma(b,u) ||^2 + \gamma^2 \cdot ||u||^2 
+ 2\gamma \IP{ 2 \Gamma(b,u) , u } \nonumber \\
& = 4 || \Gamma(b,u) ||^2 - \gamma^2 \cdot ||u||^2
\end{align}
Also,
\begin{align*} 
2 \Gamma (b,u) (x) & = \sum_y c(x,y) ( b(x) - b(y) ) (u(x) - u(y) ) \nonumber \\
& = - \sum_{y \neq \vec x}  c(x,y) (u(x) - u(y)) + c(x,\vec x) (u(x) - u(\vec x) ) \nonumber  \\
& = - \Delta u (x) + \tfrac2d (u(x) - u(\vec x)) = (\tfrac2d - \lambda) u(x) - \tfrac2d \vec u (x) , 
\end{align*}
so using \eqref{eqn:f and vec f 1} and \eqref{eqn:f and vec f 2}, assuming that $||u || =1$,
\begin{align} \label{eqn:Gamma(b,u) 2}
|| 2 \Gamma(b,u) ||^2 & = (1 - \lambda - \gamma )^2 || u ||^2 + \frac{4}{d^2} || \vec u ||^2 
- \tfrac{4}{d} (1 - \lambda - \gamma  ) \IP{ u, \vec u} \nonumber \\
& = (1-\lambda)^2 + \gamma^2 - 2 \gamma(1-\lambda) + \frac{4}{d^2} (d-1) 
- 2 (1-\lambda - \gamma) (1-\lambda ) \nonumber  \\
& = \gamma^2 + (1-\lambda_{\min})^2
- (1-\lambda )^2  .
\end{align}
Finally, 
\begin{align}
4 \Lambda(b,u) & = \sum_{x,y} c(x,y) |b(x)-b(y)|^2 \cdot |u(x) - u(y)|^2 = 2 \lambda .
\end{align}
Combining this with \eqref{eqn:Gamma(b,u) 1}, \eqref{eqn:Gamma(b,u) 2}, and plugging into 
Lemma \ref{lem:main bound}, we have that:
\begin{align*}
\sum_{i=1}^k |\lambda_{k+1} - \lambda_i |^2 \cdot (1 - \lambda_i )  &\leq 2
\sum_{i=1}^k (\lambda_{k+1} - \lambda_i) \cdot ( \lambda_i - \lambda_{\min} ) 
\cdot(1-\lambda_i +1- \lambda_{\min} ) \\
&\leq \frac{8\sqrt{d-1}}{d} \cdot
\sum_{i=1}^k (\lambda_{k+1} - \lambda_i) \cdot ( \lambda_i - \lambda_{\min} ),
\end{align*} 
where we used $\lambda_i\geq \lambda_{\min}=1-\frac{2\sqrt{d-1}}{d}.$ 
\end{proof}

\section{Applications of Yang-type inequalities}\label{sec:imp}
In this section, we derive some applications of the Yang-type inequality on graphs. 

Let $(V,c)$ be the network with the bottom of the spectrum $\lambda_{\min}.$
For any finite subset $\Omega,$ let $\{\lambda_i\}_{i=1}^{|\Omega|}$ be the Dirichlet eigenvalues of the Laplace on $\Omega.$ Set 
\begin{align} \label{eq:def:mu}
\mu_i & := \lambda_i-\lambda_{\min}\geq 0,\quad 1\leq i\leq |\Omega|.
\end{align}
By the trace of the Laplacian,
$$\sum_{i=1}^{|\Omega|}\lambda_i\leq |\Omega|.$$ Hence for any $1\leq k\leq|\Omega|,$$$\sum_{i=1}^{k}(1-\lambda_i)\geq 0.$$

\begin{cor} Suppose that the network $(V,c)$ satisfies the Yang-type inequality  \eqref{eq:Yangtype}.
Then for any finite subset $\Omega,$
$$\lambda_2-\lambda_{\min}\leq (\frac{C_{YT}}{1-\lambda_1}+1)(\lambda_1-\lambda_{\min}).$$
\end{cor}
\begin{proof} This follows from the Yang-type inequality  \eqref{eq:Yangtype} for $k=1.$
\end{proof}

The Yang-type inequality implies the following result, which is a discrete analog of Yang's second inequality.

\begin{cor}\label{cor:Yang2} Suppose that the network $(V,c)$ satisfies the Yang-type inequality  \eqref{eq:Yangtype}.
Then for any finite subset $\Omega,$ if $\lambda_{k}\leq 1+C_{YT}$ for some $1\leq k<|\Omega|,$ then 
$$\lambda_{k+1}-\lambda_{\min}\leq \frac{\sum_{i=1}^k(\lambda_i-\lambda_{\min})(1+C_{YT}-\lambda_i)}{\sum_{i=1}^k(1-\lambda_i)}.$$
\end{cor}
\begin{proof}
Let $C = C_{YT}$.
Without loss of generality, we may assume that $\lambda_{k+1}>\lambda_1,$ otherwise the result is trivial. By the Yang-type inequality \eqref{eq:Yangtype},
$$\frac1k\sum_i(\mu_{k+1}-\mu_i)\left[(\mu_{k+1}-\mu_i)(1-\mu_i-\lambda_{\min})-C\mu_i\right]\leq 0,$$ where $\{\mu_i\}_i$ is defined in \eqref{eq:def:mu} and $C=C_{YT}.$ Set $a_i:=\mu_{k+1}-\mu_i$ and $$b_i:=(\mu_{k+1}-\mu_i)(1-\mu_i-\lambda_{\min})-C\mu_i.$$ Note that the function $$f(x):=(\mu_{k+1}-x)(1-x-\lambda_{\min})-Cx$$ is non-increasing in $(-\infty,\frac12(1+C+\mu_{k+1}-\lambda_{\min})].$ Moreover, the assumption $\lambda_{k}\leq 1+C$ yields that $$\mu_i\leq \frac12(1+C+\mu_{k+1}-\lambda_{\min}),$$ which implies that $b_i$ is non-increasing. Using Chebyshev's inequality, \ie $$\sum_ia_ib_i\geq k\sum_ia_i\sum_ib_i,$$ we have
$$\left(\mu_{k+1}-\frac1k\sum_{i=1}^k\mu_i\right)\left[\mu_{k+1}\cdot\frac1k\sum_{i=1}^k(1-\lambda_i)-\frac1k\sum_{i=1}^k\mu_i(1+C-\lambda_i)\right]\leq 0.$$ Note that by $\lambda_{k+1}>\lambda_1,$ $$\lambda_{k+1}>\frac1k\sum_{i=1}^k\lambda_i.$$ Thus,
$$\mu_{k+1}\leq \frac{\sum_{i=1}^k\mu_i(1+C-\lambda_i)}{\sum_{i=1}^k(1-\lambda_i)},$$ which proves the theorem.
\end{proof}

By the above result, we derive the following inequality, a discrete analog of the Hile-Protter inequality.
\begin{cor}\label{cor:HP11} Suppose that the network $(V,c)$ satisfies the Yang-type inequality  \eqref{eq:Yangtype}.
Then for any finite subset $\Omega,$ if $\lambda_{k}\leq 1+C_{YT}$ for some $1\leq k<|\Omega|,$ then 
\begin{equation*}\sum_{i=1}^k\frac{\lambda_i-\lambda_{\min}}{\lambda_{k+1}-\lambda_i}\geq \frac{1}{C_{YT}}\sum_{i=1}^k(1-\lambda_i).\end{equation*}
\end{cor}
\begin{proof} Without loss of generality, we may assume that $\lambda_k<\lambda_{k+1}.$ Let $C=C_{YT}.$ Set $g(x):=\frac{x}{\mu_{k+1}-x},$ which is convex in $x\in(-\infty,\mu_{k+1}).$ Hence
\begin{eqnarray}\label{eq:hh1}&&\frac1k\sum_{i=1}^k\frac{\lambda_i-\lambda_{\min}}{\lambda_{k+1}-\lambda_i}=\frac1k\sum_{i=1}^k\frac{\mu_i}{\mu_{k+1}-\mu_i}\nonumber\\&=&\frac 1k\sum_{i}g(\mu_i)\geq g\left(\frac1k\sum_i{\mu_i}\right)=\frac{\frac1k\sum_i\mu_i}{\mu_{k+1}-\frac1k\sum_i\mu_i},\end{eqnarray} where we used Jensen's inequality for $g(x).$ By Corollary~\ref{cor:Yang2}, 
\begin{eqnarray*}\mu_{k+1}&\leq& \frac{\sum_{i=1}^k\mu_i(1+C-\lambda_i)}{\sum_{i=1}^k(1-\lambda_i)}\\&=&\frac{C\sum_{i=1}^k\mu_i}{\sum_{i=1}^k(1-\lambda_i)}+\frac{\sum_{i=1}^k\mu_i(1-\lambda_i)}{\sum_{i=1}^k(1-\lambda_i)}\\&\leq&\frac{C\sum_{i=1}^k\mu_i}{\sum_{i=1}^k(1-\lambda_i)}+ \frac1k\sum_{i=1}^k\mu_i,
\end{eqnarray*} where we used Chebyshev's inequality in the last line.

By plugging it into \eqref{eq:hh1}, we prove the result.
\end{proof}

This result yields a discrete analog of the Paley-Polya-Weinberger inequality.
\begin{cor}\label{cor:PPW11} Suppose that the network $(V,c)$ satisfies the Yang-type inequality \eqref{eq:Yangtype}.
Then for any finite subset $\Omega,$ if $\lambda_{k}\leq 1+C_{YT}$ for some $1\leq k<|\Omega|,$ then 
\begin{equation*}\lambda_{k+1}-\lambda_k\leq C_{YT}\frac{\sum_{i=1}^k (\lambda_i-\lambda_{\min})}{\sum_{i=1}^k(1-\lambda_i)}.\end{equation*}
\end{cor}
\begin{proof} Without loss of generality, we assume that $\lambda_k<\lambda_{k+1}.$ By Corollary~\ref{cor:HP11},
$$\frac{\sum_{i=1}^k(\lambda_i-\lambda_{\min})}{\lambda_{k+1}-\lambda_k}\leq \sum_{i=1}^k\frac{\lambda_i-\lambda_{\min}}{\lambda_{k+1}-\lambda_i}\geq \frac{1}{C_{YT}}\sum_{i=1}^k(1-\lambda_i),$$ which yields the result.
\end{proof}

We remark that for amenable groups, groups with Abelian quotients, and $d$-trees, the discrete analogs of the Paley-Polya-Weinberger inequality and the Hile-Protter inequality, as in Corollary~\ref{cor:PPW11} and Corollary~\ref{cor:HP11} without the assumption that $\lambda_{k}\leq 1+C_{YT}$ for some $1\leq k<|\Omega|,$ can be derived using same arguments in \cite[Theorem~1.1 and Theorem~1.3]{HLS17}.  

We recall a recursion formula proved by Cheng and Yang \cite{CY07}, see also \cite[Theorem~4.2]{HLS17}.
\begin{prop}\label{recursion}
Let $a_1\leq a_2\leq\cdots\leq a_{k+1}$ be any positive numbers and $\theta>0$ such that
\begin{eqnarray}\label{eq:recursion}
\sum_{i=1}^k(a_{k+1}-a_i)^2\leq\theta\sum_{i=1}^ka_i(a_{k+1}-a_i).
\end{eqnarray}
Define
$$F_k=\left(1+\frac{\theta}{2}\right)\left(\frac{1}{k}\sum_{i=1}^ka_i\right)^2-\frac{1}{k}\sum_{i=1}^ka_i^2.$$
Then we have
\begin{eqnarray*}%\label{eq:recursion-2}
F_{k+1}\leq \left(\frac{k+1}{k}\right)^{\theta}F_k.
\end{eqnarray*}

\end{prop}

Now we prove an upper bound estimate for $\lambda_k.$
\begin{cor}\label{coro:ratio}Suppose that the network $(V,c)$ satisfies the Yang-type inequality \eqref{eq:Yangtype}.
Then for any finite subset $\Omega,$ if $\lambda_k\leq 1-\delta$ for some $\delta>0,$ then
\begin{equation}\label{eq:est1}\lambda_{k+1}-\lambda_{\min}\leq \left(1+\theta\right)k^{\frac{\theta}{2}}(\lambda_1-\lambda_{\min}),\end{equation} where $\theta=\frac{1}{\delta}C_{YT}.$
\end{cor} 
\begin{proof}
Let $\mu_i:=\lambda_i-\lambda_{\min}.$ By the Yang-type inequality \eqref{eq:Yangtype}, we have 
$$\sum_{i=1}^{k} ( \mu_{k+1} - \mu_i )^2 (1- \lambda_i ) 
\leq C\cdot \sum_{i=1}^k (\mu_{k+1} - \mu_i)  \mu_i,$$ where $C=C_{YT}.$ Since $\lambda_k\leq 1-\delta,$ $1-\lambda_i\geq \delta$ for any $1\leq i\leq k.$ This yields that
\begin{equation}\label{eq:d1}\sum_{i=1}^{k} (\mu_{k+1} - \mu_i )^2  
\leq \theta\cdot \sum_{i=1}^k (\mu_{k+1} - \mu_i)  \mu_i,\end{equation} where $\theta=\frac{C}{\delta}.$ By the recursion formula in Proposition~\ref{recursion}, setting $a_i=\mu_i,$ 
$$F_{k+1}\leq \left(\frac{k+1}{k}\right)^{\theta}F_k.$$ Since the above result holds for all small $k,$ we have
$$\frac{F_{k+1}}{(k+1)^\theta}\leq \frac{F_{k}}{k^\theta}\leq \cdots \leq F_1=\frac{\theta}{2}a_1^2.$$
By \eqref{eq:d1},
$$\left(a_{k+1}-(1+\frac{\theta}{2}) A_k\right)^2\leq (1+\frac{\theta}{2})^2A_k^2-(1+\theta)B_k=(1+\theta)F_k-\frac{\theta}{2}(1+\frac{\theta}{2})A_k^2.$$
This yields that $$\frac{\theta}{2(1+\theta)}a_{k+1}^2+\left(a_{k+1}-(1+\theta)A_k\right)^2\leq (1+\theta) F_k.$$ Hence
$$a_{k+1}^2\leq \frac{2(1+\theta)^2}{\theta}F_k\leq (1+\theta)^2k^{\theta}a_1^2.$$ This proves the result.
\end{proof}

\bibliographystyle{alpha}
%\bibliography{eigenvalue_bounds}

\begin{thebibliography}{PPW56}

\bibitem[AB91]{AshBen1991}
Mark~S. Ashbaugh and Rafael~D. Benguria.
\newblock Isoperimetric bound for {$\lambda_3/\lambda_2$} for the membrane
  problem.
\newblock {\em Duke Math. J.}, 63(2):333--341, 1991.

\bibitem[AB92]{AshBen92}
Mark~S. Ashbaugh and Rafael~D. Benguria.
\newblock A sharp bound for the ratio of the first two eigenvalues of
  {D}irichlet {L}aplacians and extensions.
\newblock {\em Ann. of Math. (2)}, 135(3):601--628, 1992.

\bibitem[AB94]{AshBen1994}
Mark~S. Ashbaugh and Rafael~D. Benguria.
\newblock Isoperimetric inequalities for eigenvalue ratios.
\newblock In {\em Partial differential equations of elliptic type ({C}ortona,
  1992)}, Sympos. Math., XXXV, pages 1--36. Cambridge Univ. Press, Cambridge,
  1994.

\bibitem[AB96]{AshBen96}
Mark~S. Ashbaugh and Rafael~D. Benguria.
\newblock Bounds for ratios of the first, second, and third membrane
  eigenvalues.
\newblock In {\em Nonlinear problems in applied mathematics}, pages 30--42.
  SIAM, Philadelphia, PA, 1996.

\bibitem[AB07]{AshBen2007}
Mark~S. Ashbaugh and Rafael~D. Benguria.
\newblock Isoperimetric inequalities for eigenvalues of the {L}aplacian.
\newblock In {\em Spectral theory and mathematical physics: a {F}estschrift in
  honor of {B}arry {S}imon's 60th birthday}, volume~76 of {\em Proc. Sympos.
  Pure Math.}, pages 105--139. Amer. Math. Soc., Providence, RI, 2007.

\bibitem[Ash99]{Ash99}
Mark~S. Ashbaugh.
\newblock Isoperimetric and universal inequalities for eigenvalues.
\newblock In {\em Spectral theory and geometry ({E}dinburgh, 1998)}, volume 273
  of {\em London Math. Soc. Lecture Note Ser.}, pages 95--139. Cambridge Univ.
  Press, Cambridge, 1999.

\bibitem[Ash02]{Ash02}
Mark~S. Ashbaugh.
\newblock The universal eigenvalue bounds of {P}ayne-{P}\'olya-{W}einberger,
  {H}ile-{P}rotter, and {H}. {C}. {Y}ang.
\newblock {\em Proc. Indian Acad. Sci. Math. Sci.}, 112(1):3--30, 2002.
\newblock Spectral and inverse spectral theory (Goa, 2000).

\bibitem[BHJ14]{BHJ14}
Frank Bauer, Bobo Hua, and J\"{u}rgen Jost.
\newblock The dual {C}heeger constant and spectra of infinite graphs.
\newblock {\em Adv. Math.}, 251:147--194, 2014.

\bibitem[CC08]{ChenCheng08}
Daguang Chen and Qing-Ming Cheng.
\newblock Extrinsic estimates for eigenvalues of the {L}aplace operator.
\newblock {\em J. Math. Soc. Japan}, 60(2):325--339, 2008.

\bibitem[CG98]{CouGri98}
T.~Coulhon and A.~Grigoryan.
\newblock Random walks on graphs with regular volume growth.
\newblock {\em Geom. Funct. Anal.}, 8(4):656--701, 1998.

\bibitem[CH53]{CouHil53}
R.~Courant and D.~Hilbert.
\newblock {\em Methods of mathematical physics. {V}ol. {I}}.
\newblock Interscience Publishers, Inc., New York, N.Y., 1953.

\bibitem[Cha84]{Chavel84}
I.~Chavel.
\newblock {\em Eigenvalues in {R}iemannian geometry}, volume 115 of {\em Pure
  and Applied Mathematics}.
\newblock Academic Press, Inc., Orlando, FL, 1984.
\newblock Including a chapter by Burton Randol, With an appendix by Jozef
  Dodziuk.

\bibitem[CP13]{ChengPeng13}
Qing-Ming Cheng and Yejuan Peng.
\newblock Estimates for eigenvalues of l operator on self-shrinkers.
\newblock {\em Commun. Contemp. Math.}, 15(6):1350011, 23, 2013.

\bibitem[CY00]{ChungYau00}
Fan Chung and S.-T. Yau.
\newblock A {H}arnack inequality for {D}irichlet eigenvalues.
\newblock {\em J. Graph Theory}, 34(4):247--257, 2000.

\bibitem[CY05]{ChengYang05}
Qing-Ming Cheng and Hongcang Yang.
\newblock Estimates on eigenvalues of {L}aplacian.
\newblock {\em Math. Ann.}, 331(2):445--460, 2005.

\bibitem[CY06]{ChengYang06JPS}
Qing-Ming Cheng and Hongcang Yang.
\newblock Inequalities for eigenvalues of {L}aplacian on domains and compact
  complex hypersurfaces in complex projective spaces.
\newblock {\em J. Math. Soc. Japan}, 58(2):545--561, 2006.

\bibitem[CY07]{CY07}
Qing-Ming Cheng and Hongcang Yang.
\newblock Bounds on eigenvalues of {D}irichlet {L}aplacian.
\newblock {\em Math. Ann.}, 337(1):159--175, 2007.

\bibitem[CY09]{ChengYang09}
Qing-Ming Cheng and Hongcang Yang.
\newblock Estimates for eigenvalues on {R}iemannian manifolds.
\newblock {\em J. Differential Equations}, 247(8):2270--2281, 2009.

\bibitem[CZL12]{CZL12}
Daguang Chen, Tao Zheng, and Min Lu.
\newblock Eigenvalue estimates on domains in complete noncompact {R}iemannian
  manifolds.
\newblock {\em Pacific J. Math.}, 255(1):41--54, 2012.

\bibitem[CZY16]{CZY16}
Daguang Chen, Tao Zheng, and Hongcang Yang.
\newblock Estimates of the gaps between consecutive eigenvalues of {L}aplacian.
\newblock {\em Pacific J. Math.}, 282(2):293--311, 2016.

\bibitem[Dod84]{Dod84}
Jozef Dodziuk.
\newblock Difference equations, isoperimetric inequality and transience of
  certain random walks.
\newblock {\em Trans. Amer. Math. Soc.}, 284(2):787--794, 1984.

\bibitem[ESHI09]{EHI09}
Ahmad El~Soufi, Evans~M. Harrell, II, and Sa\"\i~d Ilias.
\newblock Universal inequalities for the eigenvalues of {L}aplace and
  {S}chr\"odinger operators on submanifolds.
\newblock {\em Trans. Amer. Math. Soc.}, 361(5):2337--2350, 2009.

\bibitem[Fri93]{Frie93}
Joel Friedman.
\newblock Some geometric aspects of graphs and their eigenfunctions.
\newblock {\em Duke Math. J.}, 69(3):487--525, 1993.

\bibitem[Har93]{Har93}
Evans~M. Harrell, II.
\newblock Some geometric bounds on eigenvalue gaps.
\newblock {\em Comm. Partial Differential Equations}, 18(1-2):179--198, 1993.

\bibitem[Har07]{Har07}
Evans~M. Harrell, II.
\newblock Commutators, eigenvalue gaps, and mean curvature in the theory of
  {S}chr\"odinger operators.
\newblock {\em Comm. Partial Differential Equations}, 32(1-3):401--413, 2007.

\bibitem[HLS17]{HLS17}
Bobo Hua, Yong Lin, and Yanhui Su.
\newblock {Payne-Polya-Weinberger, Hile-Protter and Yang's inequalities for
  Dirichlet Laplace eigenvalues on integer lattices}.
\newblock {\em arXiv:1710.05799}, 2017.

\bibitem[HM94]{HarMich94}
Evans~M. Harrell, II and Patricia~L. Michel.
\newblock Commutator bounds for eigenvalues, with applications to spectral
  geometry.
\newblock {\em Comm. Partial Differential Equations}, 19(11-12):2037--2055,
  1994.

\bibitem[HP80]{HP80}
G.~N. Hile and M.~H. Protter.
\newblock Inequalities for eigenvalues of the {L}aplacian.
\newblock {\em Indiana Univ. Math. J.}, 29(4):523--538, 1980.

\bibitem[HS97]{HarStu97}
Evans~M. Harrell, II and Joachim Stubbe.
\newblock On trace identities and universal eigenvalue estimates for some
  partial differential operators.
\newblock {\em Trans. Amer. Math. Soc.}, 349(5):1797--1809, 1997.

\bibitem[Kes59]{kesten1959full}
Harry Kesten.
\newblock Full {B}anach mean values on countable groups.
\newblock {\em Math. Scand.}, 7:146--156, 1959.

\bibitem[Kob20]{Kobayashi20}
Shinichiro Kobayashi.
\newblock {An upper bound for higher order eigenvalues of symmetric graphs }.
\newblock {\em arXiv:2006.07632}, 2020.

\bibitem[KS97]{korevaari1997global}
Nicholas~J. Korevaar and Richard~M. Schoen.
\newblock Global existence theorems for harmonic maps to non-locally compact
  spaces.
\newblock {\em Comm. Anal. Geom.}, 5(2):333--387, 1997.

\bibitem[Leu91]{Leung91}
Pui~Fai Leung.
\newblock On the consecutive eigenvalues of the {L}aplacian of a compact
  minimal submanifold in a sphere.
\newblock {\em J. Austral. Math. Soc. Ser. A}, 50(3):409--416, 1991.

\bibitem[Li80]{Li80}
Peter Li.
\newblock Eigenvalue estimates on homogeneous manifolds.
\newblock {\em Comment. Math. Helv.}, 55(3):347--363, 1980.

\bibitem[Li12]{Libook12}
Peter Li.
\newblock {\em Geometric analysis}, volume 134 of {\em Cambridge Studies in
  Advanced Mathematics}.
\newblock Cambridge University Press, Cambridge, 2012.

\bibitem[LY83]{LY83}
Peter Li and Shing~Tung Yau.
\newblock On the {S}chr\"{o}dinger equation and the eigenvalue problem.
\newblock {\em Comm. Math. Phys.}, 88(3):309--318, 1983.

\bibitem[Mok95]{mok1995harmonic}
Ngaiming Mok.
\newblock Harmonic forms with values in locally constant {H}ilbert bundles.
\newblock In {\em Proceedings of the {C}onference in {H}onor of {J}ean-{P}ierre
  {K}ahane ({O}rsay, 1993)}, number Special Issue, pages 433--453, 1995.

\bibitem[Oza18]{ozawa2018functional}
Narutaka Ozawa.
\newblock A functional analysis proof of {G}romov's polynomial growth theorem.
\newblock {\em Ann. Sci. \'{E}c. Norm. Sup\'{e}r. (4)}, 51(3):549--556, 2018.

\bibitem[P\'61]{Pol61}
G.~P\'{o}lya.
\newblock On the eigenvalues of vibrating membranes.
\newblock {\em Proc. London Math. Soc. (3)}, 11:419--433, 1961.

\bibitem[Pet17]{Pete}
G{\'a}bor Pete.
\newblock Probability and geometry on groups, 2017.
\newblock available at: \url{http://math.bme.hu/~gabor}.

\bibitem[PPW56]{PPW}
L.~E. Payne, G.~P\'{o}lya, and H.~F. Weinberger.
\newblock On the ratio of consecutive eigenvalues.
\newblock {\em J. Math. and Phys.}, 35:289--298, 1956.

\bibitem[SCY08]{SCY08}
Hejun Sun, Qing-Ming Cheng, and Hongcang Yang.
\newblock Lower order eigenvalues of {D}irichlet {L}aplacian.
\newblock {\em Manuscripta Math.}, 125(2):139--156, 2008.

\bibitem[SY94]{SY94}
R.~Schoen and S.-T. Yau.
\newblock {\em Lectures on differential geometry}.
\newblock Conference Proceedings and Lecture Notes in Geometry and Topology, I.
  International Press, Cambridge, MA, 1994.
\newblock Lecture notes prepared by W.Y. Ding, K.C. Chang [G.Q. Zhang], J.Q.
  Zhong and Y.C. Xu, Translated from the Chinese by Ding and S.Y. Cheng,
  Preface translated from the Chinese by K. Tso.

\bibitem[Tho69]{Thomp69}
Colin~J. Thompson.
\newblock On the ratio of consecutive eigenvalues in {$N$}-dimensions.
\newblock {\em Studies in Appl. Math.}, 48:281--283, 1969.

\bibitem[Wey12]{Weyl}
Hermann Weyl.
\newblock Das asymptotische {V}erteilungsgesetz der {E}igenwerte linearer
  partieller {D}ifferentialgleichungen (mit einer {A}nwendung auf die {T}heorie
  der {H}ohlraumstrahlung).
\newblock {\em Math. Ann.}, 71(4):441--479, 1912.

\bibitem[Yan91]{Yang91}
Hongcang Yang.
\newblock An estimate of the difference between consecutive eigenvalues.
\newblock {\em Preprintn IC/91/60 of ICTP, Trieste}, 1991.

\bibitem[YY80]{YangYau80}
Paul~C. Yang and Shing~Tung Yau.
\newblock Eigenvalues of the {L}aplacian of compact {R}iemann surfaces and
  minimal submanifolds.
\newblock {\em Ann. Scuola Norm. Sup. Pisa Cl. Sci. (4)}, 7(1):55--63, 1980.

\end{thebibliography}

\end{document}